\documentclass{amsart}
\usepackage{amsmath,amssymb,latexsym,tabularx}
\usepackage[dvips]{graphicx}
\newtheorem{theorem}{Theorem}[section]

\newtheorem{proposition}[theorem]{Proposition}
\newtheorem{corollary}[theorem]{Corollary}
\newtheorem{lemma}[theorem]{Lemma}
\newtheorem{remark}[theorem]{Remark}
\newtheorem{example}[theorem]{Example}

\newcommand{\cl}{\operatorname{cl}} 
\newcommand{\id}{\operatorname{id}}

\newcommand{\Hom}{\operatorname{Hom}}

\title[metric compactifications and coarse structures]{Metric compactifictions and coarse structures}

\author[K. Mine]{Kotaro Mine}
 \address[K. Mine]{Institute of Mathematics,
 University of Tsukuba, Tsukuba, 305-8571, Japan}
 \email{pen@math.tsukuba.ac.jp}

\author[A.~Yamashita]{Atsushi Yamashita}
\address[A.~Yamashita]{Graduate School of Information Sciences, Tohoku University, Sendai 980-8579, Japan}
\email{yamyam@rock.sannet.ne.jp}


\begin{document}
\maketitle
\begin{abstract}
Let $\mathbf{TB}$ be the category of totally bounded, locally compact metric spaces
with the $C_0$ coarse structures. We show that if $X$ and $Y$ are in $\mathbf{TB}$ then $X$ and $Y$ are coarsely equivalent if and only if their Higson coronas are homeomorphic. In fact, the Higson corona functor gives an equivalence of categories $\mathbf{TB}\to\mathbf{K}$, where $\mathbf{K}$ is the category of compact metrizable spaces. We use this fact to show that the continuously controlled coarse structure on a locally compact space $X$ induced by some metrizable compactification $\tilde{X}$ is determined only by the topology of the remainder $\tilde{X}\setminus X$.
\end{abstract}

\section{Introduction}
When studying ``large-scale'' or ``asymptotic'' structures of metric spaces, one is often led to consider a kind of ``boundary at infinity'' of them, for example the boundary sphere $\partial_\infty \mathbb{H}^n=S^{n-1}$ of the Poincar\'{e} ball $\mathbb{H}^n$. This boundary sphere reflects the geometry of $\mathbb{H}^n$ in the sense that the isometries of $\mathbb{H}^n$ are in one-to-one correspondence with the M\"{o}bius transformations of $S^{n-1}$. 

In many situations we can associate a boundary at infinity to a metric space, and in the optimal case, the large-scale structure in question is recovered from the boundary. Results in this direction are pursued by several authors, including Paulin~\cite{Paulin}, Bonk-Schramm~\cite{Bonk-Schramm}, Buyalo-Schroeder~\cite{Buyalo-Schroeder} and Jordi~\cite{Jordi}. As an example, let $X$ and $Y$ be Gromov hyperbolic geodesic spaces and $\partial_\infty X$ and $\partial_\infty Y$ their boundaries at infinity. We can define a visual metric on each of these boundaries, which is an analogue of the angle metric on $\partial_\infty \mathbb{H}^n=S^{n-1}$ (see~\cite[Chapter III.H]{Bridson-Haefliger}). Then, under some niceness condition (for example, it is satisfied by Cayley graphs of Gromov hyperbolic groups and their boundaries), the metric spaces $X$ and $Y$ are quasi-isometric if and only if $\partial_\infty X$ and $\partial_\infty Y$ are quasi-M\"{o}bius equivalent \cite{Buyalo-Schroeder, Jordi}.

In the present paper, we prove another such correspondence in more topological settings. Let $X$ be a locally compact, totally bounded metric space. Then, our main result states that a large-scale structure called the \textit{$C_0$ coarse structure} on $X$ introduced by Wright~\cite{Wright} (see \S 2) is completely recovered from the topology of the boundary $\tilde{X}\setminus X$, where $\tilde{X}$ stands for the completion of $X$ (Theorem \ref{category_equivalence}). 

Before introducing our results in more details, we informally review the notion of coarse structure (see \S 2 for formal definitions).  ``Large-scale'' properties of spaces, such as quasi-isometry invariant properties of finitely generated groups, can be described by coarse structures. A coarse structure on a set $X$ is given by a collection of \textit{controlled} subsets of $X\times X$ satisfying several axioms. When $E\subset X\times X$ is a fixed controlled subset, one think of $x$ and $y$ as ``close uniformly''  for all $(x,y)\in E$. Thus a typical coarse structure on a metric space $X$ is the \textit{bounded coarse structure}, where  $E\subset X\times X$ is controlled if and only if there exists $C>0$ such that $d(x,y)\leq C$ for all $(x,y)\in E$. In this structure, the phrase ``close uniformly'' above has its usual meaning.
The $C_0$ coarse structure on a locally compact metric space, mentioned above, is another kind of coarse structure. Roughly, the phrase ``close uniformly'' in the $C_0$ structure actually means ``becoming closer and closer as points approach to infinity''.

Given a suitable coarse structure on a locally compact Hausdorff space, we can define the \textit{Higson compactification} $hX$ of $X$ (see \S 2), a compactification of $X$ defined in terms of the ring of ``slowly oscillating'' functions called the Higson functions. The remainder $\nu X=hX\setminus X$ is called the \textit{Higson corona}, and $\nu X$ can be regarded as a boundary of $X$. The corona $\nu X$ is a coarse invariant, in the sense that ``coarsely equivalent'' coarse spaces have homeomorphic Higson coronas~\cite[Corollary 2.42]{Roe}. Then, it is now natural to ask whether the converse holds: if $\nu X$ and $\nu Y$ are homeomorphic, then are $X$ and $Y$ coarsely equivalent? As we mentioned earlier, an analogous statement is true for Gromov hyperbolic groups.

The paper of Cuchillo-Ib\'{a}\~{n}ez, Dydak, Koyama and Mor\'{o}n~\cite{CDKM}
gives an affirmative answer to this question about Higson coronas in some special case. They considered $Z$-sets (which are ``thin'' closed subsets in some sense) in the Hilbert cube and their complements, where each $Z$-set can be regarded as the Higson corona of the complement equipped with the $C_0$ structure. Their result then states that the category of $Z$-sets in the Hilbert cube (and the continuous maps between them) is isomorphic to the category of the $C_0$ coarse spaces formed by their complements.

In the present paper, we extend the argument in \cite{CDKM} to general locally compact metric spaces equipped with the $C_0$ structure. 
Formally stated, our main result claims an equivalence of categories $\mathbf{TB}\to \mathbf{K}$, where $\mathbf{TB}$ is the category of totally bounded locally compact metric spaces and $C_0$ coarse maps modulo closeness, and $\mathbf{K}$ is the category of compact metrizable spaces and continuous maps (Theorem \ref{category_equivalence}). This equivalence is realized by the Higson corona functor, which in this case reduces to the operation of taking the complement in the completion.  As a consequence of the equivalence $\mathbf{TB}\simeq \mathbf{K}$, it follows that the $C_0$ coarse structure on $M\setminus Z$, where $Z$ is a nowhere dense closed set in a compact metric space $M$, is determined (up to coarse equivalence) only from the topological type of $Z$, regardless of the space $M$ or how $Z$ is embedded in $M$ (Corollary \ref{topology_and_C0}).  

A compactification $\tilde{X}$ of a (locally compact Hausdorff) space $X$ in general induces a natural coarse structure on $X$, called the \textit{continuously controlled coarse structure} (see \S 2). Since this structure can be regarded as a $C_0$ coarse structure with the Higson compactification $\tilde{X}$ (see Corollary \ref{compact_metric_smirnov} and Remark \ref{another_category}), we have that the continuously controlled structure on $X$ is determined, up to coarse equivalence, by the topological type of the remainder $\tilde{X}\setminus X$ (Corollary \ref{topological_coarse}).

\section{Preliminaries on coarse strcutures and Higson coronas}
We refer the reader to Roe's monograph~\cite{Roe} as a basic reference for this section.

A \textit{coarse structure} on a set $X$ is defined as a collection $\mathcal{E}$ of subsets of $X\times X$, called \textit{controlled sets,} satisfying the following five conditions: (i) the diagonal $\Delta_X=\{(x,x)\,|\,x\in X\}$ belongs to $\mathcal{E}$, (ii) if $E\in\mathcal{E}$ and $E'\subset E$ then $E'\in\mathcal{E}$, (iii) if $E\in\mathcal{E}$ then its inverse $E^{-1}=\{(x,y)\in X\times X\,|\,(y,x)\in E\}$ belongs to $\mathcal{E}$, (iv) if $E, F\in \mathcal{E}$ then the composition $E\circ F=\{(x,z)\in X\times X\,|\,\text{there exists }y\in X\text{ such that }(x,y)\in E\text{ and }(y,z)\in F\}$ belongs to $\mathcal{E}$, and (v) if $E, F\in \mathcal{E}$ then the union $E\cup F$ belongs to $\mathcal{E}$. The pair $(X,\mathcal{E})$ (or briefly $X$) is then called a \textit{coarse space}. A subset $B\subset X$ is called \textit{bounded} in the coarse space $X$ if $B\times B$ is controlled.

Let $X$ and $Y$ be coarse spaces. We can define a class of maps from $X$ to $Y$ that respect coarse structures, namely the coarse maps, as follows. A map $f\colon X\to Y$ is called \textit{proper} if the inverse image $f^{-1}(B)$ is bounded for every bounded set $B$ of $Y$.  The map $f$ is called \textit{bornologous} if $(f\times f)(E)\subset Y\times Y$ is controlled for every controlled set $E\subset X\times X$. Then, we say that $f\colon X\to Y$ is a \textit{coarse map} if it is both proper and bornologous. A coarse map $f\colon X\to Y$ is called a \textit{coarse equivalence} if there exists a coarse map $g\colon Y\to X$ such that both $g\circ f$ and $f\circ g$ are close to their respective identities. Here maps $h, k\colon S\to Z$ from a set $S$ to a coarse space $Z$ are called \textit{close} if the set $\{(h(s), k(s))\,|\,s\in S\}$ is controlled. Coarse spaces $X$ and $Y$ are then called \textit{coarsely equivalent}.

A coarse structure on a paracompact Hausdorff space $X$ is called \textit{proper} (in which case we say that $X$ is a \textit{proper coarse space})  if (1) there is a controlled neighborhood of the diagonal $\Delta_X$ and (2) every bounded subset has compact closure. For a proper coarse space $X$, the converse statement of (2) is also true if $X$ is \textit{coarsely connected}, that is, each singleton $\{(x,y)\}$ is controlled (see \cite[Proposition 2.23]{Roe}). Notice also that a proper coarse space is necessarily locally compact.

As mentioned in the introduction, a standard example of a coarse structure is the \textit{bounded coarse structure} on a metric space $(X,d)$, where $E\subset X\times X$ is defined to be controlled if there exists $C>0$ such that $d(x,y)\leq C$ for every $(x,y)\in E$. 
In this structure, the bounded sets are exactly the bounded sets in the metric sense. The bounded coarse structure on $X$ is proper if and only if $X$ is \textit{proper} as a metric space, that is, every closed bounded subset of $X$ is compact. It is not difficult to show that two geodesic metric spaces with the bounded coarse structures are coarsely equivalent if and only if they are quasi-isometric.

For a locally compact metric space $(X,d)$, we can define a coarse structure other than the bounded structure, called the \textit{$C_0$ coarse structure} which is introduced by Wright~\cite{Wright}. In the $C_0$ coarse structure, a subset $E$ of $X\times X$ is defined to be controlled if for every $\varepsilon>0$ we can find a compact set $K\subset X$ such that $d(x,y)<\varepsilon$ for every $(x,y)\in E\setminus K\times K$. The following is proved for completeness.

\begin{proposition}\label{C0_is_proper}
Let $(X,d)$ be a locally compact metric space. Then, the above definition of the $C_0$ coarse structure indeed gives a coarse structure on $X$, where a subset is bounded if and only if it has compact closure. In case $X$ is separable, this structure is proper.
\end{proposition}

\begin{proof}
Let $(X,d)$ be a locally compact metric space.
It is easy to verify the conditions (i), (ii), (iii) and (v). To see (iv), take any controlled sets $E, F$ and $\varepsilon>0$. We prove that $E\circ F$ is also controlled. Since $E\cup F$ is controlled, we can choose a compact set $K_0$ of $X$ such that $d(x,y)<\varepsilon/2$ whenever $(x,y)\in (E\cup F)\setminus K_0\times K_0$. Since $X$ is locally compact, there is an $\varepsilon'>0$ with $\varepsilon'\leq\varepsilon/2$ such that the closed $\varepsilon'$-neighborhood $\overline{N}(K_0,\varepsilon')$ of $K_0$ is compact. Then, we can choose a compact set $K$ of $X$ containing $\overline{N}(K_0,\varepsilon')$ such that $d(x,y)<\varepsilon'$ whenever $(x,y)\in (E\cup F)\setminus K\times K$. 
 
We claim that $d(x,y)<\varepsilon$ holds for every $(x,y)\in (E\circ F)\setminus K\times K$.
Given $(x,y)\in (E\circ F)\setminus K\times K$, we can find a $z\in X$ such that $(x,z)\in E$ and $(z,y)\in F$. Since $(x,y)\notin K\times K$, either $x\notin K$ or $y\notin K$ holds. We first consider the case when $x\notin K$. Then, we see from $(x,z)\in E\setminus K\times K$ that $d(x,z)<\varepsilon'$. Since $\overline{N}(K_0,\varepsilon')\subset K$, we have $z\notin K_0$, and in particular, $(z,y)\in F\setminus K_0\times K_0$. This in turn implies that $d(z,y)<\varepsilon/2$, and hence $d(x,y)\leq d(x,z)+d(z,y)<\varepsilon'+\varepsilon/2\leq \varepsilon$. Since the case when $y\notin K$ can be treated in a similar way, the condition (iv) is verified.

It is clear from the definition of the $C_0$ coarse structure that every subset of $X$ with compact closure is bounded. To show the converse, let $B\subset X$ be a bounded set with respect to the $C_0$ structure, and suppose that $B$ does not have compact closure. Then, in particular, there are two distinct points $p, q\in B$, and we set the distance $\varepsilon=d(p,q)>0$. Since $B$ is bounded, the
square $B\times B$ is controlled, and hence there exists a compact set $K\subset X$ such that $d(x,y)<\varepsilon/2$ whenever $(x,y)\in B\times B\setminus K\times K$. Since the closure of $B$ is not compact, $B$ is not contained in $K$. Fix a point $r\in B\setminus K$ and observe that $(p,r), (q,r)\in B\times B\setminus K\times K$. This implies that $\varepsilon=d(p,q)\leq d(p,r)+d(q,r)<\varepsilon/2+\varepsilon/2=\varepsilon$, which is a contradiction. 

We further assume that $X$ is separable. To prove that the $C_0$ structure is proper, it remains only to show that there is a controlled neighborhood of the diagonal $\Delta_X$. Since $X$ is locally compact and separable metrizable, we can take a countable locally finite open cover $\{U_n\,|\;n\in\mathbb{N}\}$ such that each $U_n$ has compact closure. Then, we can define a continuous function $f\colon X\to (0,\infty)$ by
\[
f(x)=\sum_{i\in\mathbb{N}} \min\{2^{-i}, d(x, X\setminus U_i)\}.
\]
Then, it is easy to see that the function $f$ vanishes at infinity, that is, for all $\varepsilon>0$ there is a compact set $K\subset X$ such that $0<f(x)<\varepsilon$ for every $x\notin K$. This implies that the set
\[
E=\{(x,y)\in X\times X\,|\;d(x,y)<\min\{f(x), f(y)\}\}
\]
is a controlled neighborhood of $\Delta_X$. 
\end{proof}

Let $X=(X,\mathcal{E})$ be a coarse space. A bounded (not necessarily continuous) function $f\colon X\to \mathbb{R}$ is a \textit{Higson function} on $X$ if for every controlled set $E\in \mathcal E$ and $\varepsilon>0$ there is a bounded set $B\subset X$ such that $|f(x)-f(y)|<\varepsilon$ whenever $(x,y)\in E\setminus B\times B$.
The Higson functions on $X$ form a unital Banach algebra which is denoted by $B_h(X)$. 

A coarse space is usually equipped with a topology, and it makes sense to speak of continuous functions on the coarse space.  Let $X$ be a locally compact Hausdorff coarse space, and let $C_h(X)$ be the Banach algebra of \textit{continuous} Higson functions on $X$. Let $e:X\to \mathbb R^{C_h(X)}$ be an embedding into a product of lines defined by $e(x)=(f(x))_{f\in C_h(X)}$. Then, the compactification $hX=\cl_{\mathbb R^{C_h(X)}} e(X)$ of $X$ is homeomorphic to the maximal ideal space of $C_h(X)$.  We call $hX$ the \textit{Higson compactification} of $X$, and its boundary $\nu X=hX\setminus X$ is then called the \textit{Higson corona} of $X$.

The next lemma connects Higson functions and coarse maps. The proof is straightforward and left to the reader.

\begin{lemma}\label{higson_to_higson}
Let $X$ and $Y$ be locally compact Hausdorff coarse spaces satisfying the condition $(\star)$ and $f\colon X\to Y$ a coarse map. Then for every Higson function $\varphi$ on $Y$, the composition $\varphi\circ f$ is a Higson function on $X$. Consequently, $f$ induces a ring homomorphism $f^*\colon B_h(Y)\to B_h(X)$. If moreover $f$ is continuous,  $f$ induces $f^*\colon C_h(Y)\to C_h(X)$.
\qed
\end{lemma}

\begin{remark}\label{star_condition}
\normalfont
In the definition of Higson functions, we used the notion of bounded sets which is purely coarse one. In many cases a coarse space has a topology, and it is natural to assume that the bounded sets have some relation with the topology. For a locally compact Hausdorff coarse space $X$, we consider the following condition:
\[
\text{A subset of $X$ is bounded if and only if it has compact closure.}
\tag{$\star$}
\]

Hereafter we will consider the Higson corona of $X$ only when this condition is satisfied.
The condition $(\star)$ is satisfied by the following coarse structures: the bounded structures on proper metric spaces, the continuously controlled structures (defined below), the $C_0$ structures on locally compact metric spaces (Proposition \ref{C0_is_proper}), and all coarsely connected proper coarse spaces. 
\end{remark}

For a set $X$ and subsets $E\subset X\times X$ and $K\subset X$, we define $E[K]$ to be the set of $x\in X$ such that $(x,y)\in E$ for some $y\in K$. This set is the ``image'' of $K$ under $E$, where $E$ is considered to be a multivalued function from the \textit{second} coordinate to the \textit{first} coordinate. Now assume that $X$ has a topology. Then $E\subset X\times X$ is called \textit{proper} if each of $E[K]$ and $E^{-1}[K]$ has compact closure for every compact subset $K$ of $X$.

Let $X$ be a locally compact Hausdorff space with a (Hausdorff) compactification $\tilde{X}$. Denote the boundary $\tilde{X}\setminus X$ by $\partial X$. Then, since $X$ is locally compact, $X$ is open in $\tilde{X}$ and hence $\partial X$ is compact.
A subset $E\subset X\times X$ is then defined to be \textit{continuously controlled} by $\tilde{X}$ if one of (hence all of) the following three equivalent conditions is satisfied:  (a) the closure of $E$ in $\tilde{X}\times\tilde{X}$ intersects the complement of $X\times X$ only in the diagonal $\Delta_{\partial X}=\{(\omega,\omega)\,|\;\omega\in\partial X\}$, (b) $E$ is proper (in the sense defined in the previous paragraph), and for every net $\bigl((x_\lambda, y_\lambda)\bigr)$ in $E$, if $(x_\lambda)$ converges to $\omega\in\partial X$, then $(y_\lambda)$ also converges to $\omega$, (c) $E$ is proper, and for every point $\omega\in\partial X$ and every neighborhood $V$ of $\omega$ in $\tilde{X}$, there is a neighborhood $U\subset V$ of $\omega$ in $\tilde{X}$ such that $E\cap(U\times (X\setminus V))=\emptyset$. Then, the collection of all continuously controlled subsets is shown to be a coarse structure called the \textit{continuously controlled coarse structure} induced by $\tilde{X}$ (see \cite[Section 2.2]{Roe}). 

\begin{remark}\label{no_controlled_nbd}
\normalfont
For a continuously controlled structure, it is easy to see that the condition $(\star)$ is always satisfied,  while it may happen that  there is no controlled neighborhood of the diagonal, even if the space is paracompact. This means that such a structure need not be proper. (In \cite[Theorem 2.27]{Roe}, it is asserted that every continuously controlled structure on a paracompact space is proper, but the proof given there is actually incorrect, as pointed out by Berndt\nolinebreak { }Grave: see \cite{Roecorrection}.) 

As an example, let $X=[0,\infty)$ and consider the Stone-\v{C}ech compactification $\beta X$ of $X$. Let $U$ be any neighborhood of $\Delta_X$ in $X\times X$. For each $n\in\mathbb{N}$, let $a_n=n$ and take $b_n$ so that $0<b_n-a_n<2^{-1}$ and $(a_n, b_n)\in U$ are satisfied. Then $A=\{a_n\,|\,n\in\mathbb{N}\}$ and $B=\{b_n\,|\,n\in\mathbb{N}\}$ are disjoint closed subsets in $X$, and hence there exists a continuous map $f\colon X\to [0,1]$ with $f(A)=\{0\}$ and $f(B)=\{1\}$. This $f$ admits a continuous extension $\tilde{f}\colon\beta X\to [0,1]$ and we have $\cl_{\beta X} A\subset \tilde{f}^{-1}(0)$ and $\cl_{\beta X} B\subset \tilde{f}^{-1}(1)$. In particular, $\cl_{\beta X}A$ and $\cl_{\beta X} B$ are disjoint.
Since $A$ is noncompact, there exists a point $\omega\in (\cl_{\beta X} A)\setminus X$ and a net $(a_{n_\lambda})$ in $A$ convergent to $\omega$. Then the net $(b_{n_\lambda})$ has a subnet $(b_{n'_\mu})$ convergent to some point $\omega'\in \cl_{\beta X}B$. The corresponding subnet $(a_{n'_\mu})$ converges to $\omega$. Then $(a_{n'_\mu}, b_{n'_\mu})\in U$ and $(a_{n'_\mu}, b_{n'_\mu})\to (\omega, \omega')\notin \Delta_{\beta X\setminus X}$, showing that $U$ is not controlled. 
\end{remark}

In the rest of this section, we discuss how a noncontinuous coarse map between proper coarse spaces induces a continuous map between their Higson coronas.  The results will be applied to prove our main theorem (Theorem \ref{category_equivalence}).

For a proper coarse space $X$ satisfying $(\star)$, let  $B_0(X)$ denote the set of bounded, real-valued functions that vanish at infinity, in the sense that for all $\varepsilon>0$ there exists a compact set $K$ such that we have $|f(x)|<\varepsilon$ for all $x\in X\setminus K$. Let $C_0(X)$ denote the subalgebra of all continuous functions in $B_0(X)$.
The Banach algebra $C(\nu X)$ of real-valued continuous functions of the Higson corona is then isomorphic to $C_h(X)/C_0(X)$. There is a natural isomorphism $C_h(X)/C_0(X)\cong B_h(X)/B_0(X)$ by \cite[Lemma 2.40]{Roe}, and hence $C(\nu X)\cong B_h(X)/B_0(X)$. 

Now let $X$ and $Y$ be two proper coarse spaces satisfying $(\star)$ and $f\colon X\to Y$ a (not necessarily continuous) coarse map.
By Lemma \ref{higson_to_higson} there is an induced map $f^*\colon B_h(Y)\to B_h(X)$, and by the properness of $f$, we have $f^*(B_0(Y))\subset B_0(X)$. Therefore, we have a map $f^*\colon C(\nu Y)\cong B_h(Y)/B_0(Y)\to B_h(X)/B_0(X)\cong C(\nu X)$. Then, $\nu f\colon \nu X\to \nu Y$ is defined as the continuous map corresponding to the last $f^*$ by Gel'fand-Naimark duality. 
This makes the operation $\nu$ a functor, called the \textit{Higson corona functor}, from the category of proper coarse spaces to the category of compact Hausdorff spaces.

Of course, we can expect the map $\nu f$ to be a ``continuous extension'' of $f$ in some sense. In fact, we have the following:

\begin{proposition}\label{nuf}
Let $f\colon X\to Y$ be a coarse map between proper coarse spaces satisfying the condition $(\star)$. Then the map $\nu f\colon \nu X\to \nu Y$ is characterized by the property that $f\cup\nu f\colon hX\to hY$ is continuous at each point of $\nu X$.
\end{proposition}

\begin{proof}
We first show that $\nu f\colon \nu X\to\nu Y$ satisfies this property. Since $\nu f$ is continuous, we need only to show that for each  net $(x_\lambda)$ converging to a point $\omega\in \nu X$, the net $(f(x_\lambda))$ converges to $\nu f (\omega)$. If this is not the case, there exists a subnet $(x_{\lambda_\mu})$ of $(x_\lambda)$ such that $(f(x_{\lambda_\mu}))$ is convergent to $\omega'\in\nu Y\setminus \{\nu f(\omega)\}$. Then, there exists a continuous function $\tilde{\varphi}\colon hY\to \mathbb{R}$ with $\tilde{\varphi}(\nu f(\omega))=0$ and $\tilde{\varphi}(\omega')=1$, which restricts to a Higson function $\varphi=\tilde{\varphi}|_Y\in C_h(Y)\subset B_h(Y).$ Then, since $f$ is coarse, we have $\varphi\circ f\in B_h(X)$ by Lemma \ref{higson_to_higson}. Using Tietze's theorem, we can take a continuous extension $\psi\colon hX\to \mathbb{R}$ of $\tilde{\varphi}\circ (\nu f)\colon \nu X\to \mathbb{R}$. The definition of $\nu f$ yields that $\varphi\circ f-(\psi|_X)\in B_0(X)$. This implies, by the continuity of $\psi$,
\[
\lim \varphi\circ f(x_{\lambda_\mu})=\lim \psi(x_{\lambda_\mu})=\psi(\omega)=\tilde{\varphi}\circ (\nu f)(\omega)=0.
\]
On the other hand, by the continuity of $\tilde{\varphi}$,
\[
\lim \varphi\circ f(x{\lambda_\mu})=\tilde{\varphi}(\omega')=1,
\]
which is a contradiction.

The map $\nu f$ is uniquely determined by the property we have now demonstrated, since every point of $\nu X$ is a limit of some net in $X$. This completes the proof.
\end{proof}

\begin{remark}
\normalfont
The above proposition means that $\nu f$ is characterized by the fact that  $f\cup \nu f\colon (hX, \nu X)\to (hY, \nu Y)$ is \textit{eventually continuous} in the sense of  \cite[Definition 1.14]{CP} and \cite[Definition 2.4]{Rosenthal}, or is \textit{ultimately continuous} in the sense of \cite[Section 2]{HPR}. This observation is already made in the special case that  both $X$ and $Y$ are continuously controlled by some metrizable compactifications $\tilde{X}$ and $\tilde{Y}$, respectively~\cite{HPR}. In fact, the Higson compactifications $hX$ and $hY$  are equivalent to $\tilde{X}$ and $\tilde{Y}$ in this special case~\cite[Proposition 2.48]{Roe}.
\end{remark}

In some situation, it is also true that $f$ must be coarse whenever $f$ admits an extension as in the last proposition. For a precise statement we need the following notion: a map $f\colon X\to Y$ between coarse spaces is called \textit{pre-bornologous} if $f(B)\subset Y$ is bounded for every bounded set $B\subset X$. Notice that every bornologous map between coarse spaces is pre-bornologous.

\begin{proposition}\label{coarse_into_conti_controlled}
Let $X$ and $Y$ be proper coarse spaces satisfying $(\star)$ and $f\colon X\to Y$ a (not necessarily continuous) pre-bornologous map. Suppose that $Y$ has the continuously controlled coarse stucture induced by some compactification $\tilde{Y}$ of $Y$. Then, $f$ is coarse if and only if there exists $\tilde f\colon \nu X\to \nu Y$ (which is necessarily equal to $\nu f$) such that $f\cup\tilde{f} \colon hX\to hY$ is continuous at each point of $\nu X$.
\end{proposition}

\begin{proof}
The ``only if'' part is Proposition \ref{nuf}. We prove the ``if'' part. Suppose that there is a map $\tilde{f}\colon \nu X\to \nu Y$ as above. To see that $f$ is proper, it is enough to show that $f^{-1}(K)$ has compact closure in $X$ whenever $K\subset Y$ is compact, since both $X$ and $Y$ satisfy the condition $(\star)$. Let $K$ be a compact subset of $Y$. If $f^{-1}(K)$ does not have compact closure in $X$, then there exists a point $\omega\in\nu X\cap \cl_{hX} f^{-1}(K)$. Then we have $\tilde{f}(\omega)\in \nu Y$, but the continuity of $f\cup\tilde{f}$ at $\omega$ implies $\tilde{f}(\omega)\in \cl_{hY} K=K\subset Y$. This is a contradiciton, which means that $f^{-1}(K)$ has compact closure in $X$.

To prove that $f$ is bornologous, let $E$ be a controlled subset of $X\times X$ and consider the image $F=(f\times f)(E)\subset Y\times Y$. It is straightforward to show that $F$ is proper as a subset of $Y\times Y$, using the fact that $E$ is proper (see \cite[Proposition 2.23]{Roe}) and that $f$ is a proper, pre-bornologous map. Let $\bigl((f(x_\lambda), f(x'_\lambda))\bigr)$ be a net in $F$ with $(x_\lambda, x'_\lambda)\in E$ and $f(x_\lambda) \to \omega\in\tilde{Y}\setminus Y$.  
It remains to show that $f(x'_\lambda)\to\omega$.

Suppose that this is not the case. Then, there exist subnets $(x_{\lambda_\mu})$ and $(x'_{\lambda_\mu})$ (with the same index set) such that $f(x'_{\lambda_\mu})\to\omega'$ for some $\omega'\neq\omega$. We write $x_{\lambda_\mu}=x_\mu$, $x'_{\lambda_\mu}=x'_\mu$ to simplify notation. Choose a continuous function $\tilde{\varphi}\colon \tilde{Y}\to [0,1]$ such that $\tilde{\varphi}(\omega)=0$ and $\tilde{\varphi}(\omega')=1$, and let $\varphi$ denote the restriction $\tilde{\varphi}|_Y\colon Y\to [0,1]\subset \mathbb{R}$. 
By \cite[Proposition 2.45 (b)]{Roe}, there exists a continuous map $\pi\colon hY\to \tilde{Y}$ that restricts to the identity on $Y$.
Then, the composition $F=\tilde{\varphi}\circ \pi\circ (f\cup \tilde{f})\colon hX\to\mathbb{R}$ gives an extension of $\varphi\circ f$ over $hX$ which is continuous at each point in $\nu X$. By Tietze's theorem, there exists a continuous extension $G\colon hX\to\mathbb{R}$ of $\tilde{\varphi}\circ\pi\circ\tilde{f}=F|_{\nu X}$.  Then, we have $G|_X\in C_h(X)$ and $(G-F)|_X\in B_0(X)$, which in turn implies 
\[
\varphi\circ f=F|_X=G|_X-(G-F)|_X\in C_h(X)+B_0(X)=B_h(X).
\]
This causes a contradiction, since it can also be shown that $\varphi\circ f\notin B_h(X)$, as follows. Given a compact set $K\subset X$, we can take $\mu$ so large that $|\varphi\circ f(x_\mu)|<1/3$, $|\varphi\circ f(x'_\mu)-1|<1/3$, and $x_\mu\notin E[K]$. Then $x'_\mu\notin K$ and it follows that $(x_\mu, x'_\mu)\in E\setminus K\times K$ and $|\varphi\circ f(x_\mu)-\varphi\circ f(x'_\mu)|\geq 1/3$. This shows that $\varphi\circ f\notin B_h(X)$. 
\end{proof}

\section{$C_0$ and continuously controlled coarse structures}
In this section, \textit{all locally compact metric spaces are assumed to have the $C_0$ coarse structures. Controlled sets, coarse maps and Higson functions will be with respect to the $C_0$ structure.} For such structures, we first make clear how the notions of Higson functions and coarse maps are related to uniform continuity (Proposition \ref{C0_Higson=unif_conti}, Corollary \ref{coarse_char}). Then, we prove that the continuously controlled coarse structure induced by the Higson compactification is the original $C_0$ structure (Theorem \ref{C0_is_topological}).

\begin{proposition}\label{C0_Higson=unif_conti}
Let $(X,d)$ be a locally compact metric space. Then the continuous Higson functions on $X$ are exactly the bounded uniformly continuous functions on $X$.
\end{proposition}

\begin{proof}
First assume that $f\colon X\to \mathbb{R}$ is bounded and uniformly continuous. Take any controlled set $E$ in the $C_0$ structure and $\varepsilon>0$. Then, we can choose a $\delta>0$ such that $d(x,y)<\delta$ implies $|f(x)-f(y)|<\varepsilon$, and then we can choose a compact set $K$ such that $(x,y)\in E\setminus K\times K$ implies $d(x,y)<\delta$. Then, $|f(x)-f(y)|<\varepsilon$ holds for every point $(x,y)\in E\setminus K\times K$. This proves that $f$ is a Higson function.

To show the converse, suppose that $f$ is continuous but not uniformly continuous. The latter condition means that there are $\varepsilon>0$ and sequences $(x_n)_{n\in\mathbb{N}}, (x'_n)_{n\in\mathbb{N}}$ in $X$ such that $d(x_n,x'_n)<1/n$ and $|f(x_n)-f(x'_n)|\geq \varepsilon$. Then, the set $\{x_n\,|\,n\in\mathbb{N}\}$ is not contained in any compact set. Indeed, if it were contained in a compact set, then the closure of $\{x_n, x'_n\,|\,n\in\mathbb{N}\}$ would be compact, where $f$ must be uniformly continuous, contrary to the choice of $(x_n)$ and $(x'_n)$.  To show that $f$ is not a Higson function, we first notice that the set $E=\{(x_n, x'_n)\,|\,n\in \mathbb{N}\}$ is controlled, and take any compact subset $K$ of $X$. As seen above, the set $\{x_n\,|\,n\in\mathbb{N}\}$ is not contained in $K$. Thus, we can find an $N$ such that $x_N\notin K$. This means $(x_N, x'_N)\in E\setminus K\times K$, but we have also that $|f(x_N)-f(x'_N)|\geq \varepsilon$. Therefore, $f$ is not a Higson function.
\end{proof}

In what follows, we give a characterization of coarse maps between locally compact metric spaces without assuming continuity. We recall from the last section that $f\colon X\to Y$ between coarse spaces is \textit{pre-bornologous} if for every bounded $B\subset X$ the image $f(B)$ is bounded.
Since locally compact metric spaces satisfy the condition $(\star)$ in Remark \ref{star_condition} by Proposition \ref{C0_is_proper}, we obtain the following:

\begin{lemma}\label{proper_in_C0}
Let $X$ and $Y$ be locally compact metric spaces and $f\colon X\to Y$ a (not necessarily continuous) map. Then, $f$ is proper if and only if $f^{-1}(K)$ has compact closure for every compact set $K$ of $Y$. Similarly, $f$ is pre-bornologous if and only if $f(K)$ has compact closure for every compact set $K$ of $X$.
\qed
\end{lemma}

\begin{proposition}\label{characterizing_coarse_in_C0}
Let $X$ and $Y$ be locally compact metric spaces and $f\colon X\to Y$ a (not necessarily continuous) proper, pre-bornologous map. The following are equivalent:
\begin{itemize}
\item[(a)] $f$ is a coarse map.
\item[(b)] For every $\varepsilon>0$, there exist a compact set $K\subset X$ and a $\delta>0$ such that $d(f(x),f(x'))<\varepsilon$ whenever $(x,x')\notin K\times K$ and $d(x,x')<\delta$. 
\end{itemize}
\end{proposition}

\begin{proof}
(b) $\Rightarrow$ (a): Assume (b) and let $f\colon X\to Y$ be a proper, pre-bornologous map. It is enough to show that $f$ is bornologous. Take any controlled set $E\subset X\times X$ and put $F=(f\times f)(E)$. To show that $F$ is controlled, take any $\varepsilon>0$. By (b), we can take a compact set $K\subset X$ and a $\delta>0$ such that $d(x,x')<\delta$ and $(x,x')\notin K\times K$ imply $d(f(x), f(x'))<\varepsilon$. Since $E$ is controlled, there is a compact set $K'\supset K$ such that $d(x,x')<\delta$ whenever $(x,x')\in E\setminus K'\times K'$. Then, by Lemma \ref{proper_in_C0}, $L=\cl_Y f(K')$ is compact, since $f$ is pre-bornologous. Let $(y,y')\in F\setminus L\times L$. Then, $(y,y')=(f(x),f(x'))$ for some $(x,x')\in E\setminus K'\times K'$. It follows that $d(x,x')<\delta$, and hence $d(y,y')=d(f(x),f(x'))<\varepsilon$, since $(x,x')\notin K\times K$.  

(a) $\Rightarrow$ (b): Assume that $f\colon X\to Y$ is proper and pre-bornologous, and that (b) is not the case. We then prove that $f$ is not bornologous to obtain a contradiction. There exists $r>0$ such that for each $n\in\mathbb{N}$ and each compact set $K\subset X$, we can take $x_{K,n}$ and $x'_{K,n}$, not both of which are in $K$, with $d(x_{K,n}, x'_{K,n})<1/n$ and $d(f(x_{K,n}), f(x'_{K,n}))\geq r$. We may exchange $x_{K,n}$ and $x_{K',n}$ if necessary to assume that $x_{K,n}\notin K$. Fix a locally finite cover $(U_\lambda)$ of $X$ by open sets $U_\lambda$ with compact closure $D_\lambda=\cl_X U_\lambda$. Let $K_1=\emptyset$ and inductively, define $K_{n+1}$ as the union of all $D_\lambda$ that intersects $K_n\cup\{x_{K_n, n}, x'_{K_n, n}\}$. Since $(D_\lambda)$ is locally finite, we see by induction that $K_n$ is compact for each $n$. Let us define $x_n=x_{K_n, n}$ and $x'_n=x'_{K_n,n}$. Notice that $K_n\subset K_{n+1}$, $x_n\notin K_n$ and $x_n, x'_n\in K_{n+1}$.

We show that the set $E=\{(x_n,x'_n)\,|\,n\in\mathbb{N}\}\subset X\times X$ is controlled. To see this, let $\varepsilon>0$. Take $N\in\mathbb{N}$ so large that $1/N<\varepsilon$ holds, and let $K=K_N$. If $(x_n, x'_n)\notin K\times K$, then it follows that $n\geq N$, and hence $d(x_n, x'_n)<1/n\leq 1/N<\varepsilon$. This shows that $E$ is controlled.

Next, we claim that, the set $\{x_n\,|\,n\in\mathbb{N}\}$ is not contained in any compact set. Indeed, if this set is contained in a compact set, then some subsequence $(x_{n_k})$ converges to a point $x_\infty\in X$, and $D_\lambda$ is a neighborhood of $x_\infty$ for some $\lambda$. Then, for a large $k$, both $x_{n_k}$ and $x_{n_{k+1}}$ are in $D_\lambda$. Since $x_{n_k}\in D_\lambda$, we have $D_\lambda\subset K_{n_{k}+1}$. Then, $x_{n_{k+1}}\in D_\lambda\subset K_{n_{k}+1}\subset K_{n_{k+1}}$ (using $n_{k}+1\leq n_{k+1}$), which is contrary to $x_{n_{k+1}}\notin K_{n_{k+1}}$. Thus, $\{x_n\,|\,n\in\mathbb{N}\}$ is not contained in any compact set. 

Finally, we show that $(f\times f)(E)=\{(f(x_n), f(x'_n))\,|\,n\in\mathbb{N}\}$ is not controlled to prove that $f$ is not bornologous (and hence not coarse). To this end, take any compact set $K\subset Y$. Then, by Lemma \ref{proper_in_C0}, $f^{-1}(K)$ has compact closure, and hence there is some $n$ such that $x_n\notin f^{-1}(K)$ by the last paragraph, which implies $(f(x_n), f(x'_n))\notin K\times K$. However, we have 
\[
d(f(x_n),f(x'_n))=d(f(x_{K_n,n}), f(x'_{K_n,n}))\geq r.
\]
Notice that $r>0$ is irrelevant to our choice of $K$. This means $(f\times f)(E)$ is not controlled.
\end{proof}

Since continuous maps between coarse spaces satisfying $(\star)$ are pre-bornologous, and are uniformly continuous on every compact set, we obtain the following corollary:

\begin{corollary}\label{coarse_char}
A continuous map between locally compact metric spaces is coarse with respect to the $C_0$ coarse structures if and only if it is proper and uniformly continuous.
\qed
\end{corollary}

Let us consider the Higson compactification $h_0 X$ with respect to the $C_0$ structure. Then, in turn, $h_0 X$ induces a continuously controlled structure on $X$. As a generalization of \cite[Proposition 6]{CDKM}, we assert that this is the same as the original $C_0$ structure: 

\begin{theorem}\label{C0_is_topological}
The $C_0$ coarse structure on a locally compact metric space $X$ is equal to the continuously controlled structure induced by the Higson compactification $h_0 X$.
\end{theorem}

To show this theorem, the next lemma will be useful:

\begin{lemma}\label{C0controlled_criterion}
Let $X$ be a locally compact metric space and $E$ a subset of $X\times X$ with $E=E^{-1}$. Then, $E$ is controlled if and only if $d(x_n, x'_n)\to 0$ holds for every sequence $\bigl((x_n,x'_n)\bigr)_{n\in\mathbb{N}}$ in $E$ such that $(x_n)$ has no convergent subsequence.  
\end{lemma}

\begin{proof}
The ``only if'' part is clear. To show the ``if'' part, we use the construction in the proof of Proposition \ref{characterizing_coarse_in_C0} (a) $\Rightarrow$ (b), as follows. First choose a locally finite covering $(U_\lambda)_{\lambda\in\Lambda}$ of $X$ by open sets $U_\lambda$ with compact closure $D_\lambda=\cl_X U_\lambda$. Assume that $E=E^{-1}\subset X\times X$ is not controlled. Then, there exists $\varepsilon>0$ such that for each compact set $K\subset X$, we have $d(x_K, x'_K)\geq\varepsilon$ for some $(x_K, x'_K)\in E\setminus K\times K$. Here we can choose $(x_K, x'_K)$ so that $x_K\notin K$, since otherwise we can exchange $x_K$ and $x'_K$ using $E=E^{-1}$.

Let $K_1=\emptyset$, and indutively, define $K_{n+1}$ to be the union of all $D_\lambda$ that intersects $K_n\cup\{x_{K_n}\}$. Since $(D_\lambda)$ is locally finite, it follows by induction that $K_n$ is compact for each $n$. Put $x_n=x_{K_n}$ and $x'_n=x'_{K_n}$. Then, clearly, $(x_n, x'_n)\in E$.
Moreover, $(x_n)$ does not have a convergent subsequence. To see this, assume that a subsequence $(x_{n_k})$ converges to a point $x_\infty \in X$. Then, there exists a $\lambda$ such that $D_\lambda$ is a compact neighborhood of $x_\infty$. Take a large $k$ such that both of $x_{n_k}$ and $x_{n_{k+1}}$ belong to $D_\lambda$. Then $x_{n_{k+1}}\in D_\lambda\subset K_{n_k+1}\subset K_{n_{k+1}}$, which contradicts the choice of $x_{n_{k+1}}$. 
\end{proof}

The next lemma, also needed to prove Theorem \ref{C0_is_topological},  is valid for general metric spaces:

\begin{lemma}\label{r/3}
Let $(x_n)$ and $(x'_n)$ be sequences in a metric space $X$ and assume that $d(x_n, x'_n)\geq r$ for every $n\in\mathbb{N}$. Then, there exist subsequences $(x_{n_k})$ and $(x'_{n_k})$ such that $d(A, A')\geq r/3$, where
$A=\{x_{n_k}\,|\,k\in\mathbb{N}\}$ and $A'=\{x'_{n_k}\,|\,k\in\mathbb{N}\}$.
\end{lemma}

\begin{proof}
For $n\in\mathbb{N}$ define the subsets $I_n, J_n$ of $\mathbb{N}$ as follows:
\begin{gather*}
I_n=\{i\in\mathbb{N}\,|\,d(x_n, x'_i)<r/3\},\\
J_n=\{i\in\mathbb{N}\,|\,d(x_i, x'_n)<r/3\}.
\end{gather*}
Then, for $i,j\in I_n$, we have 
\[
d(x_i, x'_j)\geq r/3.
\tag{$\clubsuit$}
\]
Indeed, $d(x'_i, x'_j)\leq d(x'_i, x_n)+d(x_n, x'_j)<2r/3$, and hence $d(x_i, x'_j)\geq d(x_i, x'_i)-d(x'_i, x'_j)\geq r-2r/3=r/3$, as desired. Similarly, the inequality ($\clubsuit$) also holds for $i, j\in J_n$. Thus if $I_n$ (or $J_n$) is infinite for some $n$, the enumeration $I_n=\{n_k\,|\,k\in\mathbb{N}\}$ (or $J_n=\{n_k\,|\,k\in\mathbb{N}\}$) with $n_1<n_2<\cdots$ gives the desired subsequences $(x_{n_k})$ and $(x'_{n_k})$. We are left with the case where $I_n$ and $J_n$ are finite for all $n$.

We inductively construct a sequence $(n_k)$ which will give the desired subsequences. Let $n_1=1$, and suppose that we have constructed $n_1<\cdots <n_{k-1}$ satisfying $d(x_{n_i}, x'_{n_j})\geq r/3$ for every $i,j<k$. Notice that the set $S=\bigcup_{i<k} I_{n_i}\cup \bigcup_{i<k} J_{n_i}$ is finite. We define $n_k\in\mathbb{N}$ so that $n_k$ does not belong to $S$ and is greater than $n_{k-1}$. Then, we have $d(x_{n_k}, x'_{n_i})\geq r/3$ and $d(x_{n_i}, x'_{n_k})\geq r/3$ for each $i<k$. This completes the inductive construction.
\end{proof}

\begin{proof}[Proof of Theorem \ref{C0_is_topological}]
By \cite[Proposition 2.45 (a)]{Roe}, every $C_0$ controlled set is continuously controlled by $h_0 X$. To show the converse, let $E\subset X\times X$ be a subset continuously controlled by $h_0 X$. We may replace $E$ by $E\cup E^{-1}$ to assume that $E=E^{-1}$. To apply Lemma \ref{C0controlled_criterion} to $E$, let $\bigl((x_n, x'_n)\bigr)$ be a sequence in $E$ such that $(x_n)$ has no convergent subsequence, and suppose that $d(x_n, x'_n)\to 0$ does not hold. Then, passing to subsequences, we can find $r>0$ such that $d(x_n, x'_n)\geq r$ for every $n$. By Lemma \ref{r/3}, we can further pass to subsequences to obtain $d(A, A')\geq r/3$, where $A=\{x_n\,|\,n\in\mathbb{N}\}$ and $A'=\{x'_n\,|\,n\in\mathbb{N}\}$. Now define $\varphi\colon X\to \mathbb{R}$ by
\[
\varphi(x)=\frac{d(x,A)}{d(x,A)+d(x,A')}.
\]
Notice that $\varphi(A)=\{0\}$ and $\varphi(A')=\{1\}$. 
The function $\varphi$ is uniformly continuous and bounded, and hence is a Higson function by Proposition \ref{C0_Higson=unif_conti}. Thus, $\varphi$ admits a continuous extension $\tilde{\varphi}\colon h_0X\to\mathbb{R}$. 

On the other hand, we can take a subnet $(x_{n_\lambda})$ of $(x_n)$ such that $x_{n_\lambda}\to\omega$ for some $\omega\in h_0 X\setminus X$. Since $E$ is continuously controlled by $h_0 X$, we have $x'_{n_\lambda}\to\omega$. However, we then obtain
\[
0=\lim \varphi(x_{n_\lambda})=\tilde{\varphi}(\omega)=\lim \varphi(x'_{n_\lambda})=1,
\]
which is a contradiction.
\end{proof}

\section{$C_0$ coarse structures on totally bounded spaces}

The \textit{Smirnov compactification} $uX$ of a metric space $X$ is defined as the maximal ideal space of the unital Banach algebra $C_u(X)$ of real-valued bounded uniformly continuous functions. Thus, a bounded continuous function $\varphi\colon X\to\mathbb{R}$ is extendable continuously over $uX$ if and only if it is uniformly continuous, and any compactification with this property is equivalent to $uX$. Here, two compactifications $\gamma X$ and $\delta X$ of a space $X$ are called \textit{equivalent} if there exists a homeomorphism $h\colon \gamma X\to\delta X$ such that $h|_X=\id$. Proposition \ref{C0_Higson=unif_conti} immediately implies the following:

\begin{proposition}\label{smirnov=higson}
For any locally compact metric space $X$, the Smirnov compactification $uX$ of $X$ is equivalent to the Higson compactification of $X$ with respect to the $C_0$ coarse structure. \qed
\end{proposition}

On the other hand, there is a useful characterization of the Smirnov compactification of a general metric space:

\begin{theorem}\cite[Theorem 2.5]{Woods}\label{smirnov_criterion}
Let $\gamma X$ be a (Hausdorff) compactification of a metric space $X=(X,d)$. Then, $\gamma X$ is equivalent to the Smirnov compactification $uX$ if and only if $\cl_{\gamma X} A\cap \cl_{\gamma X}B=\emptyset$ for all subsets $A, B\subset X$ with $d(A, B)>0$. \qed
\end{theorem}

\begin{corollary} \label{compact_metric_smirnov}
For any compact metric space $X=(X,d)$ and its dense subspace $Y$, the space $X$ coincides with the Smirnov compactification $uY$. If moreover $Y$ is locally compact (or equivalently, open in $X$), then the $C_0$ structure on $Y$ coincides with the continuously controlled structure induced from $X$, and $X$ is the Higson compactification for this structure.
\end{corollary}

\begin{proof}
The first half of the statement is immediate from Theorem \ref{smirnov_criterion}. If $Y$ is locally compact, we can consider the $C_0$ structure on $Y$ with respect to the metric $d$ induced from $X$, as well as the continuously controlled structure on $Y$ induced by $X$. Then, by Proposition \ref{smirnov=higson}, $X=uY$ is the Higson compactification of $Y$ for the $C_0$ structure. Finally, it follows from Theorem \ref{C0_is_topological} that the continuously controlled structure on $Y$ induced by $X=uY$ is equal to the $C_0$ structure.
\end{proof}

To state our main result, we define two categories.
Let $\mathbf{K}$ be the category of compact metrizable spaces and continuous maps. We define another category $\mathbf{TB}$ as follows:  the objects of $\mathbf{TB}$ are totally bounded locally compact metric spaces with the $C_0$ coarse structures. The set $\Hom_\mathbf{TB}(X,Y)$ of morphisms between objects $X$ and $Y$ consists of the equivalence classes of coarse maps by the equivalence relation~$\sim$, where $f \sim g$ if $f$ and $g$ are close (that is, $\{(f(x),g(x))\,|\,x\in X\}$ is a controlled set). Such a category can be defined, since the closeness relation is compatible with composition from left and right.  

\begin{remark}\label{another_category}
\normalfont
The category $\mathbf{TB}$ is related to continously controlled structures. Indeed, as seen from Corollary \ref{compact_metric_smirnov}, the category $\mathbf{TB}$ is equivalent to the following category $\mathbf{CC}$: the objects of $\mathbf{CC}$ are the locally compact spaces with the continuously controlled structures induced by metrizable compactifications, and the morphisms between them are the coarse maps modulo closeness. 

On the other hand, Cuchillo-Ib\'{a}\~{n}ez, Dydak, Koyama and Mor\'{o}n~\cite{CDKM} considered the category $\mathcal{Z}$ of $Z$-sets in the Hilbert cube $Q$ and continuous maps, and they have shown that $\mathcal{Z}$ is isomorphic to the category $\mathcal{C}_0(\mathcal{Z})$ of the complements of $Z$-sets in $Q$ with the $C_0$ coarse structures and coarse maps modulo closeness (here $Q$ is assumed to have a fixed metric). Since every compact metrizable space is homeomorphic to some $Z$-set in $Q$, the category $\mathbf{K}$ is equivalent to $\mathcal{Z}$. It follows that the categories $\mathbf{K}, \mathcal{Z}$ and $\mathcal{C}_0(\mathcal{Z})$  are equivalent to each other. The next Theorem \ref{category_equivalence} implies that they are equivalent to $\mathbf{TB}$, and hence to $\mathbf{CC}$.
\end{remark}

Let us consider the Higson corona functor $\nu$ introduced before Proposition \ref{nuf}. This functor sends close coarse maps to the same continuous map (see \cite[Proposition 2.41]{Roe}), and thus coarsely equivalent proper coarse spaces have homeomorphic Higson coronas. Naturally, we can ask the converse, namely whether $X$ and $Y$ are coarsely equivalent if $\nu X$ and $\nu Y$ are homeomorphic.  This question has a negative answer in general (see~\cite[Example 2.44, Proposition 2.45 (c)]{Roe}), but the next theorem states that we have an affirmative answer for objects of $\mathbf{TB}$. 

If $X$ is an object of $\mathbf{TB}$, then the completion $\tilde{X}$ of $X$ is compact  since $X$ is totally bounded. By Corollary \ref{compact_metric_smirnov}, $\tilde{X}$ is the Higson compactification of $X$ and $\tilde{X}\setminus X$ is the Higson corona. In particular, $\nu X$ is compact and metrizable. Thus, we can define a functor $\nu\colon \mathbf{TB}\to \mathbf{K}$.

\begin{theorem}\label{category_equivalence}
The functor $\nu\colon \mathbf{TB}\to \mathbf{K}$ is an equivalence of categories.
\end{theorem}

\begin{proof}
It is enough to show that $\nu$ is full and faithful, and that every object in $\mathbf{K}$ is isomorphic to $\nu X$ for some object $X$ in $\mathbf{TB}$. 

We shall first show that $\nu$ is full, namely that $\nu$ gives a surjective map from $\Hom_\mathbf{TB}(X,Y)$ to the set $\Hom_\mathbf{K}(\nu X,\nu Y)$ of continuous maps from $\nu X$ to $\nu Y$, for each $X$ and $Y$ in $\mathbf{TB}$. Let $h\colon \nu X\to \nu Y$ be a continuous map. Recall that the completion $\tilde{X}$ of $X$ gives the Higson compactification $hX=X\cup\nu X$ of $X$, and the same holds for $hY$. Thus we use the notation $\tilde{X}$ and $\tilde{Y}$ rather than $hX$ and $hY$, and their metrics extended from $X$ and $Y$ are denoted by $d$ when necessary.

We construct (a representative of) a morphism $f\colon X\to Y$ in $\mathbf{TB}$ such that $\nu f=h$. The basic idea here is as follows: for $x\in X$, we take a point $a\in \nu X$ close to $x$ and define $f(x)$ to be a point of $Y$ close to $h(a)$, to the same extent as $x$ is close to $a$. We explain this construction in detail. Let us define $U_n$ as the open $1/n$-neighborhood of $\nu X$ in $\tilde{X}$ for $n\in\mathbb{N}$, and let $U_0=\tilde{X}$. Using the compactness of $\nu Y$, for each $n\in\mathbb{N}$, take finitely many points $y_{n,1}$, $y_{n,2}$,\ldots, $y_{n, k(n)}$ in $Y$ such that $\nu Y\subset\bigcup_{i=1}^{k(n)} B(y_{n,i}, 1/i)$. For convenience, let $k(0)=1$ and let $y_{0,0}$ be an arbitrarily fixed point in $Y$.

%

To define $f\colon X\to Y$, let $x\in X$ and take the largest $n\geq 0$ such that $x\in U_n$. If $n=0$, then we define $f(x)=y_{0,0}$. If $n\geq 1$, choose $x'\in \nu X$ such that $d(x, x')=d(x, \nu X)$. Then we can choose $i\in\{1,2,\ldots, k(n)\}$ such that $h(x')\in B(y_{n,i}, 1/n)$. We finally define $f(x)=y_{n,i}\in Y$.  

We claim that $f\colon X\to Y$ is a coarse map and $\nu f=h$. First, notice that $f$ is pre-bornologous, since $C_0$ coarse structures satisfy the condition $(\star)$ in Remark \ref{star_condition} and $f(X\setminus U_n)$ is contained in the finite set $\{y_{m,i}\,|\,m<n, 1\leq i\leq k(m) \}$ for each $n\in\mathbb{N}$. By Theorem \ref{C0_is_topological}, the $C_0$ coarse structure on $Y$ is the continously controlled structure induced by $\tilde{Y}$. Also, we easily see that $f\cup h \colon X\cup \nu X=\tilde{X}\to\tilde{Y}$ is continuous at each point in $\nu X$. Then, it follows by Proposition \ref{coarse_into_conti_controlled} (and Proposition \ref{nuf}) that $f$ is coarse and $\nu f=h$. The fullness of $\nu$ is now proved.

Next we show that $\nu\colon \mathbf{TB}\to\mathbf{K}$ is faithful, namely that $\nu$ maps each $\Hom_\mathbf{TB}(X,Y)$ injectively to $\Hom_\mathbf{K}(\nu X, \nu Y)$. To see this, let $f, g\colon X\to Y$ be coarse maps such that $\nu f=\nu g$. We have to show that $f$ and $g$ are close, in other words, $E=\{(f(x),g(x))\,|\,x\in X\}\subset Y\times Y$ is controlled. By Theorem \ref{C0_is_topological}, it is enough to show that $E$ is continuously controlled by $\tilde{Y}$. To this end, take any $(\eta, \eta')\in\overline{E}\setminus Y\times Y$, where $\overline{E}$ denotes the closure of $E$ in $\tilde{Y}\times \tilde{Y}$. Then, there exists a net $(x_\lambda)$ in $X$ such that $(f(x_\lambda), g(x_\lambda))\to (\eta, \eta')$. Since $f$ is proper, we can take a subnet $(x_{\lambda_\mu})$ of $(x_\lambda)$ such that $x_{\lambda_\mu}\to \omega$ for some $\omega\in \nu X=\tilde{X}\setminus X$. Then by Proposition \ref{nuf}, we have $\eta=\lim f(x_{\lambda_\mu})=\nu f(\omega)=\nu g(\omega)=\lim g(x_{\lambda_\mu})=\eta'\in \nu Y=\tilde{Y}\setminus Y$, which shows that $E$ is continuously controlled by $\tilde{Y}$.

Finally, we have to show that every object in $\mathbf{K}$ is isomorphic to $\nu X$ for some object $X$ in $\mathbf{TB}$. To see this, let $K$ be any compact metrizable space, and fix any admissible metric $d$ on $K\times [0,1]$. Let $X=K\times (0,1]$. Then, $X=(X,d)$ is an object of $\mathbf{TB}$ and $K\times [0,1]$ is its Higson compactification by Corollary \ref{compact_metric_smirnov}. It follows that $\nu X=K\times\{0\}$ and hence $K$ is homeomorphic to $\nu X$. The proof is completed.\end{proof}

The following is an immediate consequence of Theorem \ref{category_equivalence} (and Corollary \ref{compact_metric_smirnov}):

\begin{corollary}\label{topology_and_C0}
Suppose that $M_1$ and $M_2$ are compact metric spaces and that $Z_1 \subset M_1$ and $Z_2 \subset M_2$ are closed nowhere dense subspaces. Then, $M_1\setminus Z_1$ and $M_2\setminus Z_2$ are coarsely equivalent as $C_0$ coarse spaces if and only if $Z_1$ and $Z_2$ are homeomorphic.  
\qed
\end{corollary}

Moreover, Theorem \ref{category_equivalence} and the above corollary translate to the language of the category $\mathbf{CC}$ introduced in Remark \ref{another_category}, in view of Corollary \ref{compact_metric_smirnov}:

\begin{corollary}\label{topological_coarse}
The Higson corona functor $\nu\colon \mathbf{CC}\to \mathbf{K}$ is an equivalence of categories. In particular, two metrizable compactifications $\tilde{X}_1$ and $\tilde{X}_2$ of a locally compact space $X$ determine coarsely equivalent continuously controlled coarse structures if and only if their remainders are homeomorphic, $\tilde{X}_1\setminus X \approx \tilde{X}_2\setminus X$.\qed
\end{corollary}

\begin{corollary}\label{contractible}
Every object in $\mathbf{CC}$ is coarsely equivalent to an object in $\mathbf{CC}$ that is contractible, whose Higson compactification is also contractible.
\end{corollary}

\begin{proof}
For any object $X$ in $\mathbf{CC}$, which has the continuously controlled structure induced by a metrizable compactification $\tilde{X}$, consider the remainder $Z=\tilde{X}\setminus X$. Let $\tilde{Y}$ be the cone over $Z$, which is compact metrizable and is a compactification of the open cone $Y=\tilde{Y}\setminus Z$. We can then equip $Y$ with the continuously controlled structure induced by $\tilde{Y}$. By Corollary \ref{topological_coarse}, the coarse space $Y$ is an object of $\mathbf{CC}$ coarsely equivalent to $X$. Clearly $Y$ and $\tilde{Y}$ are contractible, and $\tilde{Y}$ is the Higson compactification of $Y$ by Corollary \ref{compact_metric_smirnov}.
\end{proof}

\begin{example}
\normalfont
Applying Corollary \ref{topological_coarse}, we can construct three proper coarse structures $\mathcal{E}_i\,(i=1,2,3)$ on the same topological space $X$ with $\mathcal{E}_1\subset \mathcal{E}_2\subset \mathcal{E}_3$ for which $\mathcal{E}_1$ and $ \mathcal{E}_3$ are coarsely equivalent, but $\mathcal{E}_2$ fails to be equivalent to $\mathcal{E}_1$ (or $\mathcal{E}_3$). Indeed, it suffices to take three metrizable compactifications $\gamma_i X$ of the same locally compact space $X$ that admit maps $\gamma_1 X\to\gamma_2 X\to\gamma_3 X$ extending the identity, with the remainders $Z_i=\gamma_i X\setminus X$ satisfying $Z_1\approx Z_3$ but $Z_1\not\approx Z_2$. Then, the continuously controlled strutcures induced by $\gamma_i X\,(i=1,2,3)$ give an example. It is easy to construct an explicit example where $X=[0,1]\times[0,1)$, $Z_2$ is a circle and $Z_1, Z_3$ are arcs.
\end{example}

We conclude this paper with results concerning embeddings of $C_0$ coarse spaces, stating that there is a ``universal'' $C_0$ coarse space in which all object in $\mathbf{TB}$ can be embedded. We say that  a map $f\colon X\to Y$ between coarse spaces is a \textit{coarse embedding} if the map $f\colon X\to f(X)$ is a coarse equivalence. Here $f(X)$ is assumed to have the induced coarse structure $\{F\subset f(X)\times f(X)\,|\,F\text{ is controlled in }Y\}$. The proof of the next lemma is straightforward:

\begin{lemma}\label{induced}
Let $X$ be a locally compact metric space with the $C_0$ coarse structure and $Y\subset X$ be a closed set. Then, the induced coarse structure on $Y$ coincides with the $C_0$ structure for the locally compact metric space $Y$.
\qed
\end{lemma}

First, we consider coarse embeddings that are topological embeddings at the same time:

\begin{proposition}\label{topological_universal}
There exists a separable locally compact metric space $X$ such that for every object $Y$ in $\mathbf{TB}$ admits a map $f\colon Y\to X$ that is simultaneously a topological and coarse embedding.
\end{proposition}

\begin{proof}
We can take $X=Q\times[0,1)$, where $Q=[0,1]^\mathbb{N}$ is the Hilbert cube. We define a metric on $X$ as the restriction of any compatible metric on $Q\times[0,1]$. Let $Y$ be any object of $\mathbf{TB}$ and $\tilde{Y}$ be its completion. We fix a continuous function $\varphi\colon \tilde{Y}\to [0,1]$ such that $\varphi^{-1}(1)=\tilde{Y}\setminus Y$ and a topological embedding $j\colon \tilde{Y}\to Q$. Then, the map $i\colon \tilde{Y}\to Q\times[0,1]$ defined by $i(y)=(j(y),\varphi(y))$ gives a topological embedding such that $i^{-1}(X)=i^{-1}(Q\times[0,1))=Y$. Let us show that $f=i|_Y\colon Y\to X$ is the required map. The maps $f\colon Y\to f(Y)$ and $f^{-1}\colon f(Y)\to Y$ are proper since they are homeomophisms, and are uniformly continuous since they are restrictions of continuous maps, namely $i$ and $i^{-1}$, defined on compact metric spaces. We conclude from Corollary \ref{coarse_char} that $f\colon Y\to X$ is a coarse embedding.
\end{proof}

If we admit coarse embeddings that are not topological embeddings (and not even continuous maps), we have the following result by using Theorem \ref{category_equivalence}:

\begin{theorem}\label{universal}
For every noncompact locally compact separable metrizable space $X$, there exists a compatible totally bounded metric $d$ on $X$ such that every object in $\mathbf{TB}$ can be coarsely embedded into $(X,d)$ with respect to the $C_0$ structure.
\end{theorem}

Corollary \ref{contractible} turns every object in $\mathbf{TB}$ into a contractible space, which is ``continuous'' in nature. The next corollary of Theorem \ref{universal} is a result in the opposite direction, saying that every object in $\mathbf{TB}$ can be expressed as a discrete metric space. Here, a \textit{discrete} metric space means a metric space whose topology is discrete.

\begin{corollary}\label{discrete}
There exists a countable discrete metric space $X$ such that every object in $\mathbf{TB}$ can be coarsely embedded into $X$ with respect to the $C_0$ structures. Moreover, every object in $\mathbf{TB}$ is coarsely equivalent to some countable discrete metric space with the $C_0$ structure.
\end{corollary}

\begin{proof}
The first part readily follows from Theorem \ref{universal}. The second part follows from the first part using Lemma \ref{induced}.
\end{proof}

To prove Theorem \ref{universal} (and Corollary \ref{discrete}), we need some technical lemmas:

\begin{lemma}\label{dense}
Let $X$ be a locally compact separable metric space with the $C_0$ coarse structure and $A, B\subset X$ with the induced structures, where $\cl_X A=B$. Then, the inclusion $A\to B$ is a coarse equivalence.
\end{lemma}

\begin{proof}
Let $i\colon A\to B$ be the inclusion, which is clearly a coarse map. By Proposition \ref{C0_is_proper}, there exists a controlled neighborhood $E_0$ of the diagonal $\Delta_X$ in $X\times X$. We define $h\colon B\to A$ by choosing a point $h(b)\in A$ with $(b,h(b))\in E_0$ for each $b\in B$. It is easy to check that $h\colon B\to A$ is also a coarse map. Then, $i\circ h$ is close to the identity $\id_B$ since the set $\{(b, h(b))\,|\,b\in B\}$ is contained in $E_0$ and hence is controlled. Similarly, the other composition $h\circ i$ is close to the identity $\id_A$. We conclude that $i\colon A\to B$ is a coarse equivalence.
\end{proof}

\begin{remark}
\normalfont
Clearly, this lemma is true for a coarse space $X$ equipped with a topology for which there is a controlled neighborhood of the diagonal $\Delta_X$ in $X\times X$, in particular for all proper coarse spaces.
Furthermore, if $X$ is such a coarse space, subsets $A$ and $B$ of $X$ are coarsely equivalent with respect to the induced structures whenever they have the same closure, $\cl_X A=\cl_X B$. 
\end{remark}

\begin{lemma}\label{embedding}
Let $X, Y$ be spaces in $\mathbf{TB}$ and $\nu X$, $\nu Y$ be their Higson coronas with respect to the $C_0$ structures. Let $j\colon \nu X\to \nu Y$ be a topological embedding. Then, there exists a coarse embedding $f\colon X\to Y$ such that $\nu f=j$.
\end{lemma}

\begin{proof}
Let $\tilde{Y}=Y\cup \nu Y$ be the Higson compactification which coincides with the completion. As shown in the proof of Theorem \ref{category_equivalence}, there exists a coarse map $f\colon X\to Y$ such that $\nu f=j$. Let $\overline{f(X)}$ be the closure of $f(X)$ in $Y$. By Proposition \ref{nuf}, it is easy to see that $\cl_{\tilde Y}\overline{f(X)}=\cl_{\tilde{Y}}f(X)=\overline{f(X)}\cup j(\nu X)$. Hence by Corollary \ref{compact_metric_smirnov}, we have $j(\nu X)=\nu \overline{f(X)}$. Let $f_0\colon X\to \overline{f(X)}$ and $j_0\colon \nu X\to j(\nu X)$ be the maps which are equal to $f$ and $j$ respectively, with their ranges restricted. Then we have $\nu f_0=j_0$ by Proposition \ref{nuf}. Notice that $\overline{f(X)}$ is closed in $Y$, hence its $C_0$ structure coincides with the structure induced from $Y$ by Lemma \ref{induced}. Since $j_0$ is a homeomorphism, $f_0\colon X\to \overline{f(X)}$ is a coarse equivalence by Theorem \ref{category_equivalence}. The coarse equivalence $f_0$ factors as $X\to f(X)\to \overline{f(X)}$, and the second map is a coarse equivalence by Lemma \ref{dense}. Then, it easily follows that the first map $X\to f(X)$ is a coarse equivalence, as desired.
\end{proof}

\begin{proof}[Proof of Theorem \ref{universal}]Recall that every compact metrizable space can be embedded into $Q=[0,1]^\mathbb{N}$. In view of Lemma \ref{embedding}, to prove this theorem it is enough to notice that there exists a metrizable compactification $\gamma X$ of $X$ with the remainder homeomorphic to $Q$.  Then the restriction to $X$ of any compatible metric on $\gamma X$ satisfies our requirement (then, $\gamma X$ is the Higson compactification with respect to the $C_0$ structure by Corollary \ref{compact_metric_smirnov}). For completeness, we explain how to construct $\gamma X$. Since $X$ is noncompact and metrizable, there exists a sequence $(x_n)_{n\in\mathbb{N}}$ of distinct points in $X$ without convergent subsequences. Fix a countable dense subset $\{y_n\,|\,n\in\mathbb{N}\}$ in $Q$. The map $\{x_n\,|\,n\in\mathbb{N}\}\to Q=[0,1]^\mathbb{N}$ that sends each $x_n$ to $y_n$ can be extended to a continuous map $h\colon X\to Q$ by Tietze's theorem. Let $K$ be the product $(X\cup\{\infty\})\times Q$, where $X\cup\{\infty\}$ denotes the one-point compactification. The map $i\colon X\to K$ defined by $i(x)=(x,h(x))$ is a topological embedding, and the closure of its image in $K$ is $i(X)\cup(\{\infty\}\times Q)$, which is clearly a metrizable compactification of $X$ with the remainder homeomorphic to $Q$.
\end{proof}

\section*{acknowledgements}
The second author deeply acknowledges the hospitality of
Oberwolfach mathematical research institute in Germany,
where he has undertaken a part of this work.
The authors thank Takamitsu Yamauchi for a useful comment, which
substantially improved the proof of Lemma \ref{r/3}.

\end{document}